\documentclass{owrart}

\usepackage{graphicx,amssymb,amsfonts,amsmath,amsthm,newlfont}
\usepackage{epsfig}
\newtheorem{thm}{Theorem}

\theoremstyle{definition}

\newtheorem{defns}[thm]{Definitions}
\newtheorem{conj}[thm]{Conjecture}

\newtheorem*{claim*}{Claim}
\theoremstyle{remark}

\font \rus= wncyr10

\newcommand{\C}{\mathbb C}

\newcommand{\R}{\mathbb R}

\newcommand{\Q}{\mathbb Q}

\newcommand{\sha}{\, \hbox{\rus x} \,}

\newcommand{\Mod}{\mathfrak{M}}
\newcommand{\M}{\overline{\mathfrak{M}}}

\newcommand{\Sym}{\mathfrak{S}}

\begin{document}
\noindent Oberwolfach Reports\\
Volume 4, Issue 2, Report 26 (2007)\\
pp. 1495-1497

\vspace{1cm}
\begin{talk}[Francis Brown, Leila Schneps]{Sarah Carr}{Periods on the
    moduli space of genus 0 curves}

\noindent  A recent theorem in the thesis of Francis Brown proves that
any period
over a connected component of
the real part of $\Mod_{0,n}(\C)$ is a $\Q$ linear
combination of
multizeta values.
By
studying the cohomology and geometry of
$\Mod_{0,n}(\C)=\Mod_{0,n}$, we have found a method to formally
represent these periods
as linear combinations of pairs of $n$-polygons, one
of which represents a connected component, or {\it cell}, of the real part of
$\Mod_{0,n}$, and the
other a certain differential form which we call a {\it cell form}.  These
pairs of polygons form an algebra 
for the shuffle product.
In this talk, we will outline the combinatorial structure of this algebra.  As
consequences, we obtain an explicit basis for the cohomology
group, $H^{n-3}(\Mod_{0,n}^\delta)$, of differential forms converging on the
boundary divisors which bound standard associahedron, $\delta$, and
hence a new method for studying
multizeta values.

We denote by $(0,t_1,...,t_{n-3},1,\infty)$ a point in $\Mod_{0,n}$
and by $\Mod_{0,n}(\R)\subset
\Mod_{0,n}$ the points whose marked points are in $\R$.
We can identify an oriented $n$-gon to a connected component, or {\it
  cell}, in $\Mod_{0,n}(\R)$ by
labelling the $n$-gon with the marked points.  This $n$-gon is
associated to the
cell given by the clockwise cyclic ordering of the
labelled edges of the polygon.  For example, a polygon cyclically labelled
$(0,t_1,t_3,1,t_2,\infty)$ is identified with the cell
$0<t_1<t_3<1<t_2<\infty$ in $\Mod_{0,n}(\R)$.

Similarly we can associate an $n$-gon labelled by the marked points
to a differential $(n-3)$-form which we call a {\it cell form}.  A cell
form is defined as
$dt_1\wedge ... \wedge dt_{n-3}/\Pi
(s_i-s_{i-1})$, where the $s_i$ are the cyclically labelled sides of
the polygon. 
We leave the side labelled $\infty$ out of the product.  For example,
the polygon cyclically
labelled $[0,1,t_1,t_3,\infty,t_2]$ gives the cell form $dt_1dt_2dt_3
/((t_1-1)(t_3-t_1) (-t_2))$.

We consider a {\it period} on $\Mod_{0,n}$ to be convergent
integral, over a connected component in $\Mod_{0,n}(\R)$, of a
differential form which is holomorphic on the interior of $\Mod_{0,n}$
and which
has at most logarithmic singularities on $\M_{0,n}\setminus \Mod_{0,n}$. 
Up to a variable change corresponding to permuting the marked points,
all periods can be written as integrals over the standard cell,
$\delta := 0<t_1<...<t_{n-3}<1$.

According to the above definitions, we can associate a pair of
polygons to a cell and a cell form.  Therefore, we have a map from
pairs of $n$-gons to periods (and divergent integrals) given by
mapping the pair to the integral of the cell form over the cell.
This association and Brown's thesis have allowed us to prove some
results and approach 
some conjectures about multizeta values and formal multizeta values.

\begin{thm} The {\it 01-cell forms} given by polygons
  $[...,0,1,...,\infty]$ form 
  a basis for
  $H^{n-3}(\Mod_{0,n})$ of differential $(n-3)$-forms convergent on
  the interior of $\Mod_{0,n}$ and with at most logarithmic
  singularities on the boundary divisors, $\M_{0,n}\setminus \Mod_{0,n}$.
\end{thm}
To prove this, it was enough to express Arnol'd's well-known basis
in terms of 01-cell forms.

\begin{defns}
Let ${\cal P}_n$ be the vector space generated by oriented $n$-gons
decorated by the marked points in $\Mod_{0,n}$.

Recall that the shuffle product of lists $A$ and $B$ is defined as 
$$A\sha B = \sum_{\sigma\in \Sym} \sigma(A\cdot B),$$
where $\sigma$ runs over the permutations of the concatenation of $A$
and $B$ such that the orders of $A$ and $B$ are preserved.

Let $I_n\subset {\cal
  P}_n$ be the vector subspace generated by shuffle sums with respect
to $\infty$, in other words polygon sums of the form $$\sum_{W\in A\sha
  B} [W,\infty]$$
where $A$, $B$ is a partition of $\{ 0,t_1,...,t_{n-3},1\}$.
\end{defns}

\begin{thm}
${\cal P}_n/I_n$ is isomorphic to $H^{n-3}(\Mod_{0,n})$. 
\end{thm}
\begin{proof} (Sketch)
By the previous theorem, we have a natural surjective map 
\begin{equation}\label{map}{\cal P}_n\twoheadrightarrow
  H^{n-3}(\Mod_{0,n}) ,\end{equation}
which sends a polygon to its
associated cell form.  A calculation on rational functions shows that
$I_n$ is in the kernel of this map. A dimension count finishes the proof.
\end{proof}
We would like to study cohomology of interesting subspaces of $
H^{n-3}(\Mod_{0,n})$ such as the space of differential forms which
converge on the standard cell, $\delta$.  To do this
we make use
of the kernel $I_n$ to create a basis of convergent 01-forms and what
we refer to as {\it
  insertion forms}.  

Some 01-forms naturally converge on $\delta$. We define a {\it chord} on a
cell form, $\omega$, to be a set of marked points of a subsequence
on $\omega$ of the length between $2$ and $\lfloor
\frac{n}{2} \rfloor$.  The 01-forms which do not have any chords in
common with $\delta$ converge on $\delta$. 
However, some linear combinations of nonconvergent 01-forms converge
on $\delta$; a certain generating set of these are {\it insertion forms}.
Insertion forms are 
created according to a recursive procedure of inserting convergent
shuffles (those whose shuffle factors have no chords in common with $\delta$)
into convergent 01-forms.  For example,
\begin{align}
\omega &=[0,1,t_1,t_2,\infty, t_3]  +[0,1,t_2,t_1,\infty, t_3]\\
&=[0,1,t_1\sha t_2, \infty,t_3]\end{align}
is an insertion form obtained by inserting the convergent shuffle
$t_1\sha t_2$ into the convergent 01-form
$[0,1,s_1,\infty,s_2]$.  The shuffle factors are $t_1$ and $t_2$, and
are therefore
too short to contain any chords.

\begin{thm}
The insertion forms and the convergent 01-cell forms form a basis for
$H^{n-3}(\Mod_{0,n}^\delta)$.
\end{thm}

The proof of this theorem is the heart of our recent work and is given
in [3].  It exploits the fact that $I_n$ is the kernel
of the map \eqref{map}.

Now that we can explicitly describe the differential forms convergent
on $\delta$,
we can define an algebra of periods, since by a variable change, all
periods can be written as integrals of forms over $\delta$.  The
algebra of periods 
has three known sets of relations:
\begin{enumerate}
\item invariance under the symmetric group action corresponding to a
  variable change;
\item forms given by shuffles with respect to one point are
  identically 0;
\item product map relations coming from the pullback of maps on
moduli space (outlined in [2] and [3]).
\end{enumerate}

The product map relations also allow us to define a multiplication law
on periods.

We conjecture that these are the only relations on periods, but this
question seems difficult to prove.  A more strategic approach is to
define a formal algebra on polygon pairs satisfying these and only
these relations.  Since the algebra of periods is isomorphic to the
algebra of multizeta values ([2]), we conjecture that the formal
algebra of pairs of polygons, which we call $\cal{ 
FC}$, is isomorphic to the formal multizeta value algebra.  With this
association, we hope to approach some of the main conjectures on
formal multizetas such as Zagier's dimension conjecture.

\begin{conj}[Zagier]  Let ${\cal Z}_n$ be the vector space generated
  by weight $n$ multizeta values.  Then $d_n=\mathrm{dim}({\cal Z}_n)$
  is given by the recursive formula,
$$d_n=d_{n-2}+d_{n-3}.$$
\end{conj}
This conjecture is true in small weight for $\cal{FC}$ and we hope that
its combinatorial recursive definition will allow us to make progress
on this conjecture.


\end{talk}
\end{document}